\documentclass[12pt,reqno]{article}
\usepackage[usenames]{color}
\usepackage{graphics}
\usepackage{amscd}
\usepackage{color}
\usepackage{float}
\usepackage{amsthm}
\usepackage{makeidx}
\usepackage{latexsym}
\usepackage{graphicx}
\usepackage{float}
\usepackage{epsfig}
\usepackage{graphics,amsmath,amssymb}

\newtheorem{theorem}{Theorem}
\newtheorem{corollary}[theorem]{Corollary}
\newtheorem{lemma}[theorem]{Lemma}
\newtheorem{proposition}[theorem]{Proposition}
\newtheorem{rem}[theorem]{Remark}

\newtheorem{defin}[theorem]{Definition}
\newenvironment{definition}{\begin{defin}\normalfont\quad}{\end{defin}}
\newtheorem{examp}[theorem]{Example}
\newenvironment{example}{\begin{examp}\normalfont\quad}{\end{examp}}

\newcommand{\seqnum}[1]{\href{http://www.research.att.com/cgi-bin/access.cgi/as/~njas/sequences/eisA.cgi?Anum=#1}{\underline{#1}}}

\usepackage[pdfpagelabels,plainpages=false]{hyperref}
\usepackage[latin1]{inputenc}

\begin{document}

\begin{center}
\vskip 1cm{\LARGE\bf Identities Involving Zeros of Ramanujan and Shanks Cubic Polynomials }
\vskip 1cm
\large
Stefano Barbero, Umberto Cerruti, Nadir Murru\\
Department of Mathematics \\
University of  Turin \\
via Carlo Alberto 10, 10123 Turin, Italy\\
\href{mailto:stefano.barbero@unito.it}{\tt stefano.barbero@unito.it}\\
\href{mailto:umberto.cerruti@unito.it}{\tt umberto.cerruti@unito.it}\\
\href{mailto:nadir.murru@unito.it}{\tt nadir.murru@unito.it}\\
\ \\
Marco Abrate\\
DIMEAS, Polytechnic University of Turin\\
Corso Duca degli Abruzzi 24, 10129 Turin, Italy\\

\href{mailto:marco$\_$abrate@polito.it}{\tt marco$\_$abrate@polito.it}\\
\end{center}

\begin{abstract}
In this paper we highlight the connection between Ramanujan cubic polynomials (RCPs) and a class of polynomials, the Shanks cubic polynomials (SCPs), which generate cyclic cubic fields. In this way we provide a new characterization for RCPs and we express the zeros of any RCP in explicit form, using trigonometric functions. Moreover, we observe that a cyclic transform of period three permutes these zeros. As a consequence of these results we provide many new and beautiful identities. Finally we connect RCPs to Gaussian periods, finding a new identity, and we study some integer sequences related to SCPs .
\end{abstract}

\section{Ramanujan and Shanks polynomials}
\begin{definition} A \emph{Ramanujan cubic polynomial} (RCP) $x^3+px^2+qx+r$ is a cubic polynomial, with $p,q,r \in \mathbb R$, $r\not = 0$, which has real zeros and with coefficients satisfying the relation
\begin{equation} \label{cond1-rcp} pr^{\frac{1}{3}}+3r^{\frac{2}{3}}+q=0. \end{equation}
\end{definition}
Shevelev \cite{She} first introduced the definition of RCP, in honour of the great mathematician Srinivasa Ramanujan, since such a class of polynomials arises from a theorem proved by Ramanujan \cite{Rama}. Witula  \cite{Witula} gave the characterization for the RCPs showed in the next theorem.

\begin{theorem}All RCPs have the form: 
\begin{multline*}
x^3 - \frac{P(\gamma-1)}{(\gamma-1)\, (\gamma-2)}\, r^{\frac{1}{3}}\, x^2 -
\frac{P(2-\gamma)}{(1-\gamma)\, (2-\gamma)}\, r^{\frac{2}{3}}\, x  + r = {}\\
{}=
\left( x - \frac{r^{\frac{1}{3}}}{2-\gamma} \right)\,
\left( x - (\gamma -1)\, r^{\frac{1}{3}}  \right)\,
\left( x - \frac{2-\gamma}{1-\gamma}\,r^{\frac{1}{3}} \right),
\end{multline*}
where $r\in \mathbb{R}\setminus \{0\}$,
$\gamma\in \mathbb{R}\setminus \{1,2\}$, and
$$
P(\gamma)=
\gamma^3 - 3\, \gamma + 1 =
\left( \gamma - 2\, \cos \frac{2\, \pi}{9} \right)\,
\left( \gamma - 2\, \cos \frac{4\, \pi}{9} \right)\,
\left( \gamma - 2\, \cos \frac{8\, \pi}{9} \right).
$$
\end{theorem}
 We provide a different representation for the RCPs in the following
\begin{theorem}
The polynomial $x^3+px^2+qx+r$ is a RCP if and only if there exist $h,s \in \mathbb R$, $s\not=0$, such that
\begin{equation} \label{cond2-rcp} r=s^3,\quad q=-(h+3)s^2,\quad p=hs.  \end{equation}
\end{theorem}
\begin{proof}
If equations \eqref{cond2-rcp} hold, then it is straightforward to check equation \eqref{cond1-rcp}. Vice--versa, if equation \eqref{cond1-rcp} holds, then for $h=\frac{p}{r^{\frac{1}{3}}}$ and $s=r^{\frac{1}{3}}$ we obtain equations \eqref{cond2-rcp}.
\end{proof}
\begin{corollary}
The polynomial $x^3+px^2+qx+r$ is a RCP if and only if it has the form
$$\rho(h,s,x)=x^3+hsx^2-(h+3)s^2x+s^3,$$
where $h,s\in\mathbb R$, $s\not=0$.
\end{corollary}
\begin{definition}
We call \emph{Shanks cubic polynomials} (SCPs) the polynomials
\begin{equation} \label{shanks} \rho(h,-1,x)=x^3-hx^2-(h+3)x-1,  \end{equation}
since Shanks \cite{Sha} deeply studied the cyclic cubic fields generated by the polynomials \eqref{shanks}.
\end{definition}
 In the following theorem, we provide an explicit expression for the zeros of SCPs using trigonometric functions. Then we use this result and the connection between SCPs and RCPs in order to obtain a closed expression for the RCPs zeros.
\begin{lemma} \label{newsolve}
Let us consider the cubic polynomial $Q(x)=ax^3+bx^2+cx+d$ and the polynomial $P(w)=w^2+fx-\frac{e^3}{27}$, where $e=\frac{1}{a}\left(c-\frac{b^2}{3a}\right), f=\frac{1}{a}\left(d+\frac{2b^3}{27a^2}-\frac{bc}{3a}\right)$. If the polynomial $P(w)$ has two complex zeros $\alpha\pm i\beta$, then the polynomial $Q(x)$ has the zeros
$$-\cfrac{b}{3a}+2 \sqrt[3]{\sqrt{\alpha^2+\beta^2}}\cos \left(\frac{1}{3}\left(\arctan \left(\frac{\beta}{\alpha}\right)+k\pi\right)\right),\quad k=0,2,4,\quad \text{if} \ \alpha>0$$
$$-\cfrac{b}{3a}-2 \sqrt[3]{\sqrt{\alpha^2+\beta^2}}\cos \left(\frac{1}{3}\left(\arctan \left(\frac{\beta}{\alpha}\right)+k\pi\right)\right),\quad k=0,2,4,\quad \text{if} \ \alpha<0.$$
\end{lemma}
\begin{proof}
It is well known that the polynomial $Q(x)$ becomes the quadratic polynomial $P(w)$ by means of the Vieta substitutions $x=-\frac{b}{3a}+z+\frac{\mu}{z}$, $z^3=w$, where $\mu=-\frac{e}{3}$. Thus from the zeros $\alpha\pm i\beta$ of $P(w)$ we find the zeros of $Q(x)$ using the relations
\begin{equation} \label{roots} -\cfrac{b}{3a}+z_i+\cfrac{s}{z_i},\quad i=1,\ldots,6 \end{equation}
where 
$$z_1=\sqrt[3]{\alpha+i\beta}, \ z_2=\omega\sqrt[3]{\alpha+i\beta}, \ z_3=\omega^2\sqrt[3]{\alpha+i\beta},$$ $$z_4=\sqrt[3]{\alpha-i\beta}, \ z_5=\omega\sqrt[3]{\alpha-i\beta}, \ z_6=\omega^2\sqrt[3]{\alpha-i\beta},$$
and the number $\omega$ represents one of the complex roots of unity. Obviously, in equations \eqref{roots} some $z_i$ determine the same values. Therefore we find exactly three different zeros for $Q(x)$. Now ,without loss of generality, we focus on $z_1$, whose complex exponential form is  
$$z_1=\sqrt[3]{\rho} e^{i\frac{\theta}{3}},$$
where $\rho=\sqrt{\alpha^2+\beta^2}$ and $\theta=\arctan\left(\frac{\beta}{\alpha}\right)$.
Furthermore, observing that  $s=\sqrt[3]{\rho^2}$, from  equations \eqref{roots} we have
$$-\cfrac{b}{3a}+2\sqrt[3]{\rho}\cos \left(\frac{\theta}{3}\right)=-\cfrac{b}{3a}\pm 2 \sqrt[3]{\sqrt{\alpha^2+\beta^2}}\cos\left( \frac{1}{3}\left(\arctan \left(\frac{\beta}{\alpha}\right)\right)\right),$$
and plus or minus sign depends on whether $\alpha>0$ or $\alpha<0$, respectively. In a similar way we obtain the remaining zeros, proving the thesis.
\end{proof}
\begin{theorem} \label{root-scp}
The zeros of the SCP $\rho(h,-1,x)$ are 
$$\frac{1}{3} \left(h+2 \sqrt{\tau(h)} \cos\left(\frac{1}{3} \left(\arctan\left(\frac{3 \sqrt{3}}{3+2 h}\right)+k\pi \right)\right)\right) \quad \text{if}\quad h\geq-\frac{3}{2},$$
$$\frac{1}{3} \left(h-2 \sqrt{\tau(h)} \cos\left(\frac{1}{3} \left(\arctan\left(\frac{3 \sqrt{3}}{3+2 h} \right)+k\pi\right)\right)\right)\quad \text{if}\quad h\leq-\frac{3}{2},$$
where $\tau(h)=h^2+3h+9$ and $k=0,2,4$. In particular, if $h=-\frac{3}{2}$, we have $\arctan(+\infty)=\frac{\pi}{2}$ in the first formula and $\arctan(-\infty)=-\frac{\pi}{2}$ in the second formula, retrieving the roots $1$,$-\frac{1}{2}$,$-2$ of the SCP $\rho(-\frac{3}{2},-1,x)$.
\end{theorem}
\begin{proof}
If $h=-\frac{3}{2}$, the proof is straightforward.
If $h>-\frac{3}{2}$, we apply lemma \ref{newsolve}. In this case we have $e=\frac{\tau(h)}{3}$, $f=-\frac{(2h+3)\tau(h)}{27}$. The zeros of $P(x)$ are 
$$\alpha\pm i\beta=\frac{1}{54}((3+2h)\tau(h)\pm 3i\sqrt{3}\tau(h))$$
with $\alpha>0$ and
$$\alpha^2+\beta^2=\cfrac{(3+2h)^2\tau(h)^2+27\tau(h)^2}{54^2}=\cfrac{\tau(h)^3}{3^6},\quad \cfrac{\beta}{\alpha}=\cfrac{3\sqrt{3}}{3+2h}\ .$$
Similarly, using lemma \ref{newsolve}, we can prove the case $h<-\frac{3}{2}$, when $\alpha<0$.
\end{proof}
\begin{theorem} \label{root-rcp}
If we consider the zeros set $\zeta(h,s)$ of the RCP $\rho(h,s,x)$ and $\alpha$ is one of these zeros, then
\begin{equation}\label{eta}\zeta(h,s)=\left\{\alpha,\cfrac{s^2}{s-\alpha},-\cfrac{s(s-\alpha)}{\alpha}\right\}=\{\alpha,\eta_s(\alpha),\eta_s^2(\alpha)\},\end{equation}
where $\eta_s(z)=\cfrac{s^2}{s-z}$.
Moreover, for any RCP $\rho\left(\frac{p}{r^{\frac{1}{3}}},r^{\frac{1}{3}},x\right)=x^3+px^2+qx+r$  we have
\begin{equation}\label{zeta}\zeta\left(\frac{p}{r^{\frac{1}{3}}},r^{\frac{1}{3}}\right)=-r^{\frac{1}{3}}\zeta\left(\cfrac{p}{r^{\frac{1}{3}}},-1\right).\end{equation}
\end{theorem}
\begin{proof}
Let us observe that if $\alpha$ is a zero of $\rho(h,s,x)$, then $\alpha\not=0$, $\alpha\not=s$  and the transform $\eta_s$ is periodic with period 3, satisfying $\eta_s^3(z)=z$. We easily check that
$$\rho(h,s,\eta_s(z))=-\cfrac{s^3\rho(h,s,z)}{(s-z)^3},\quad \rho(h,s,\eta_s^2(z))=-\cfrac{s^3\rho(h,s,z)}{z^3}.$$
So we have
$$\rho(h,s,\eta_s(\alpha))=-\cfrac{s^3\rho(h,s,\alpha)}{(s-\alpha)^3}=0,\  \rho(h,s,\eta_s^2(\alpha))=-\cfrac{s^3\rho(h,s,\alpha)}{\alpha^3}=0.$$
We also point out that for any SCP 
$$\zeta(h,-1)=\left\{\alpha,-\cfrac{1}{\alpha + 1},-\cfrac{(\alpha + 1)}{\alpha}\right\} ,\quad \eta_{-1}(z)=-\cfrac{1}{z + 1}.$$
The zeros of $\rho\left(\frac{p}{r^{\frac{1}{3}}},-1,x\right)=x^3-\frac{p}{r^{\frac{1}{3}}}x^2-\left(\frac{p}{r^{\frac{1}{3}}}+3\right)x-1$ are the elements of the set $\zeta\left(\frac{p}{r^{\frac{1}{3}}},-1\right)$. Thus, if  $\rho\left(\frac{p}{r^{\frac{1}{3}}},-1,\alpha\right)=0$, we have
$\zeta\left(\frac{p}{r^{\frac{1}{3}}},-1\right)=\left\{\alpha,-\cfrac{1}{\alpha+1},-\cfrac{\alpha+1}{\alpha}\right\}.$
Furthermore we easily prove that $\rho\left(\frac{p}{r^{\frac{1}{3}}},r^{\frac{1}{3}},-r^{\frac{1}{3}}\alpha\right)=0$, in fact
$$-r\alpha^3+pr^{\frac{2}{3}}\alpha^2-qr^{\frac{1}{3}}\alpha+r=-r\left(\alpha^3-\cfrac{p}{r^{\frac{1}{3}}}\alpha^2+\left(\cfrac{p}{r^{\frac{1}{3}}}+3\right)\alpha-1\right)=0.$$
Therefore we obtain
\begin{eqnarray*}
\zeta\left(\frac{p}{r^{\frac{1}{3}}},r^{\frac{1}{3}}\right)&=&\left\lbrace-r^{\frac{1}{3}}\alpha,\eta_{r^{\frac{1}{3}}}\left(-r^{\frac{1}{3}}\alpha\right),\eta_{r^{\frac{1}{3}}}^2\left(-r^{\frac{1}{3}}\alpha\right)\right\rbrace=\\
 &=&\left\{-r^{\frac{1}{3}}\alpha,\cfrac{r^{\frac{2}{3}}}{r^{\frac{1}{3}}+r^{\frac{1}{3}}\alpha},\cfrac{r^{\frac{1}{3}}\left(r^{\frac{1}{3}}+r^{\frac{1}{3}}\alpha\right)}{-r^{\frac{1}{3}}\alpha}\right\}=\\
 &=&-r^{\frac{1}{3}}\left\{\alpha,-\cfrac{1}{\alpha+1},-\cfrac{\alpha+1}{\alpha}\right\}=-r^{\frac{1}{3}}\zeta\left(\frac{p}{r^{\frac{1}{3}}},-1\right).
\end{eqnarray*}
\end{proof}
In the next corollary we higlight some interesting consequences of the previous theorems.
\begin{corollary}\label{cor1}
Recalling that $h=\cfrac{p}{r^{\frac{1}{3}}}$, $s=r^{\frac{1}{3}}$ and $\tau(h)=h^2+3h+9$,   
\begin{enumerate}
\item \begin{equation}\label{zeta1} \zeta(h,s)=-s\zeta(h, 1);\end{equation}
\item the zeros of any RCP $\rho(h,s,x)$  have the following expressions in terms of $\cos$ and $\arctan$ functions
\begin{equation}\label{rcp-zeros} -\cfrac{s}{3}\left(h\pm2\sqrt{\tau(h)}\cos\left(\frac{1}{3}\left(\arctan\left(\cfrac{3\sqrt{3}}{3+2h}\right)+k\pi\right)\right)\right)\quad k=0,2,4; \end{equation}
\item if we consider a zero $\alpha$ of the RCP $\rho(h,s,x)\in \mathbb Q[X]$, then the Galois group of the extension $\mathbb Q(\alpha)$ over $\mathbb Q$ has three elements: the identity automorphism and the automorphisms $\eta_s$ and $\eta^2_s$.
\end{enumerate}
\end{corollary}
\begin{proof}
The equality (\ref{zeta1}) directly follows from relation (\ref{zeta}) and allows us to prove 
equations (\ref{rcp-zeros}), using the closed expression for the zeros of any SCP given by theorem \ref{root-scp} . Finally, the relation (\ref{eta}) clearly shows that the identity automorphism and the automorphisms $\eta_s$ and $\eta^2_s$ are the only elements of the Galois group of the extension $\mathbb Q(\alpha)$, since the transform $\eta_s$ is periodic with period 3.
\end{proof}

\section{Identities}
With the expressions for the RCPs zeros, we can derive some interesting and unexpected identities. Witula, Slota and Warzynski \cite{quasifib} showed that $2\cos\left(\cfrac{2\pi}{7}\right),$ $2\cos\left(\cfrac{4\pi}{7}\right)$ and $2\cos\left(\cfrac{8\pi}{7}\right)$ are the zeros of the RCP $\rho(-1,-1,x)=x^3+x^2-2x-1$ ($h=-1, s=-1$). Thus, for example, we find the following identity
$$2\cos\left(\frac{2\pi}{7}\right)=\frac{1}{3} \left(-1+2 \sqrt{7} \cos\left(\frac{1}{3} \arctan\left(3 \sqrt{3}\right)\right)\right).$$ 
 
Many other identities may result from previous theorems observing that a RCP with zero $\alpha$ is
\begin{eqnarray} \label{rcp-extended}
\rho(h,s,x) &=& x^3+hsx^2-(h+3)s^2x+s^3 =\nonumber \\
 &=&(x-\alpha)\left(x-\frac{s^2}{s-\alpha }\right)\left(x+\frac{s (s-\alpha )}{\alpha }\right) =\nonumber\\ 
& =& x^3+\frac{s^3-3 s^2 \alpha+\alpha^3}{(s-\alpha) \alpha}x^2+\frac{-s^4+3 s^2 \alpha^2-s \alpha^3}{(s-\alpha) \alpha}x+s^3. \end{eqnarray}
Thanks to these equalities we can easily express any real number $\alpha$ in terms of $\cos$ and $\arctan$ functions . Indeed, choosing a nonzero real number $s\not=\alpha$ and $h=\frac{s^3-3s^2\alpha+\alpha^3}{s(s-\alpha)\alpha}$, we create a RCP, with one of the zeros equal to $\alpha$, on which we use the results of 
theorems \ref{root-scp} and \ref{root-rcp} .
\begin{example}
We select $\alpha=\sqrt{2}$ and $s=3\sqrt{2}$. In this way, $h=\frac{1}{6}$ and $\sqrt{2}$ is a zero of $\rho(\frac{1}{6},3\sqrt{2},x)$. By theorem \ref{root-scp} we obtain the zeros of $\rho(\frac{1}{6},-1,x)$ and multiplying these zeros by $-3\sqrt{2}$ we get the zeros of $\rho(\frac{1}{6},3\sqrt{2},x)$. We approximately evaluate these zeros (with any mathematical software) in order to equal the correct one to $\sqrt{2}$. In this case we find
$$\sqrt{2}=-\frac{1+14\sqrt{7}\cos\left(\frac{1}{3}\left(\arctan\left(\frac{9\sqrt{3}}{10}\right)+4\pi\right)\right)}{3\sqrt{2}},$$
and with further simplifications we have the identity
$$1=\sqrt{7}\cos\left(\frac{1}{3}\arctan\left(\frac{9\sqrt{3}}{10}\right)\right)-\sqrt{21}\sin\left(\frac{1}{3}\arctan\left(\frac{9\sqrt{3}}{10}\right)\right).$$
\end{example} 
\begin{example}
We consider $\alpha=\pi$ and $s=1$. Hence, $h=\frac{1-3\pi+\pi^3}{(1-\pi)\pi}<-\frac{3}{2}$ and $\pi$ is a zero of $\rho(h,1,x)$. Using theorem \ref{root-scp} we immediately have the zeros of $\rho(h,-1,x)$ and multiplying them by -1 we obtain the  zeros of $\rho(h,1,x)$. We approximately evaluate these zeros (with any mathematical software) in order to equal the correct one to $\pi$. In this case, after some algebraic calculations, we find
$$\frac{1-3 \pi +\pi ^3+2 \left(1-\pi +\pi ^2\right)^{\frac{3}{2}} \cos\left(\frac{1}{3} \arctan\left(\frac{3 \sqrt{3} (-1+\pi ) \pi }{2-3 \pi -3 \pi ^2+2 \pi ^3}\right)\right)}{3 (-1+\pi ) \pi }=\pi.$$
 From the previous equality  we derive  the nontrivial identity
$$\cfrac{2\pi-1}{2\sqrt{\pi^2-\pi+1}}=\cos\left(\frac{1}{3} \arctan\left(\frac{3 \sqrt{3} (-1+\pi ) \pi }{2-3 \pi -3 \pi ^2+2 \pi ^3}\right)\right).$$
\end{example} 
A wonderful and famous Ramanujan identity on RCPs is 
\begin{equation} \label{rama} x_1^{\frac{1}{3}}+x_2^{\frac{1}{3}}+x_3^{\frac{1}{3}}=\left(-p-6r^{\frac{1}{3}}+3(9r-pq)^{\frac{1}{3}}\right)^{\frac{1}{3}} \end{equation}
where $x_1,x_2,x_3$ are the RCP zeros (see, e.g., Berndt \cite[p.\ 22]{Br}). Using our expression \eqref{rcp-extended} and theorem \ref{root-rcp}, we rewrite identity (\ref{rama}) as
\begin{multline} \label{iden1}
\alpha^{\frac{1}{3}}=-\left(\frac{s^2}{s-\alpha}\right)^{\frac{1}{3}}-\left(-\frac{s(s-\alpha)}{\alpha}\right)^{\frac{1}{3}}{}+\\
{}+\left(\frac{s^3+3 s^2 \alpha-6 s \alpha^2+\alpha^3}{-s \alpha+\alpha^2}+3\left(s^2-s \alpha+\alpha^2\right)\left(\frac{s }{\alpha^2 (s-\alpha)^2}\right)^{\frac{1}{3}}\right)^{\frac{1}{3}}.
\end{multline}

If we choose suitable real values for $\alpha$ and $s$,  equation \eqref{iden1} generates infinite nontrivial identities.
\begin{example}
For $\alpha=\pi^3$ and $s=1$, we obtain
$$\pi = \frac{1}{\left(-1+\pi ^3\right)^{1/3}}-\frac{\left(-1+\pi ^3\right)^{1/3}}{\pi }+\left(\frac{3 \left(1-\pi ^3+\pi ^6\right)}{\pi ^2 \left(-1+\pi ^3\right)^{2/3}}+\frac{1+3 \pi ^3-6 \pi ^6+\pi ^9}{-\pi ^3+\pi ^6}\right)^{1/3}.
$$
\end{example}

\section{Gaussian periods and Ramanujan--Shanks polynomials}

\begin{definition}
Let $m$ be a positive integer and $p$ a prime number satisfying the condition $p\equiv 1$ (mod $m$). We consider the multiplicative group $G_{m,p}$ of the $m$--th powers mod $p$ whose order is $\frac{p-1}{m}$. The \emph{Gaussian periods} are
$$\eta_{m,p}^{(k)}=\sum_{j\in C_k}e^{\frac{2j\pi i}{p}},$$
where $C_0=G_{m,p}$ and for $k\geq 1$ $C_k$ are the cosets of $G_{m,p}$ in $\mathbb Z_p^*$. 
\end{definition}
\begin{example}
If $p=13$ and $m=3$, we have $G_{3,13}=C_0=\{1,5,8,12\}$, $C_1=\{2,3,10,11\}$ and $C_2=\{4,6,7,9\}$. The cubic Gaussian periods are
$$\eta_{3,13}^{(0)}=e^{-\frac{2\pi}{13}i}+e^{\frac{2\pi}{13}i}-e^{-\frac{3\pi}{13}i}-e^{\frac{3\pi}{13}i}=2\cos\left( \cfrac{2\pi}{13} \right)- 2\cos\left( \cfrac{3\pi}{13} \right),$$
$$\eta_{3,13}^{(1)}=e^{-\frac{4\pi}{13}i}+e^{\frac{4\pi}{13}i}+e^{-\frac{6\pi}{13}i}+e^{\frac{6\pi}{13}i}=2\cos\left( \cfrac{4\pi}{13} \right)+ 2\cos\left( \cfrac{6\pi}{13} \right),$$
$$\eta_{3,13}^{(2)}=e^{-\frac{8\pi}{13}i}+e^{\frac{8\pi}{13}i}+e^{-\frac{12\pi}{13}i}+e^{\frac{12\pi}{13}i}=2\cos\left( \cfrac{8\pi}{13} \right)+ 2\cos\left( \cfrac{12\pi}{13} \right).$$
\end{example}

In the following, we focus on cubic Gaussian periods, i.e., we set $m=3$ and we write $\eta_p^{(k)}=\eta_{3,p}^{(k)}$. Lehmer \cite{Le} connected the Shanks polynomials zeros with cubic Gaussian periods. In particular, when $h\in\mathbb Z$ and $3\not|h$, we call \emph{Shanks prime} the value of $\tau(h)=h^2+3h+9$ when it is a prime number.
\begin{rem} The sequence of Shanks primes in OEIS \cite{oeis} is \seqnum{A005471}:
$$(7, 13, 19, 37, 79, 97, 139,\ldots).$$
\end{rem} 
When $\tau(h)=p$ is a Shanks prime, the elements of $\zeta(h,-1)$  are fundamental units of $\mathbb Q(\alpha)$, differing from the cubic Gaussian periods by the integer number $\cfrac{L-1}{6}$, where $L=\pm(2h+3)$. Moreover, the minimal polynomial of $\eta_p^{(0)},\eta_p^{(1)},\eta_p^{(2)}$ is
\begin{equation} \label{minimal} x^3+x^2-\cfrac{p-1}{3}x-\cfrac{(L+3)p-1}{27}. \end{equation}
The sign of the integer $L$ depends on the parameter $h$. If $h\equiv 1$ (mod $3$), then $L=-(2h+3)$ and $\cfrac{(L+3)p-1}{27}$ is an integer. If $h\equiv 2$ (mod $3$), then $L=2h+3$. Thus, we can write the polynomial \eqref{minimal} as
$$G_1(h,x)=x^3+x^2-\cfrac{p-1}{3}x+\cfrac{2hp+1}{27},\quad \text{if} \ h\equiv 1 \ (\text{mod}\ 3),$$
$$G_2(h,x)=x^3+x^2-\cfrac{p-1}{3}x-\cfrac{(6+2h)p+1}{27},\quad \text{if} \ h\equiv 2 \ (\text{mod}\ 3).$$
Furthermore, the SCP $\rho(h,-1,x)$ has zeros
\begin{equation} \label{p1}  \cfrac{h-1}{3}-\eta_p^{(k)},\quad k=0,1,2\quad \text{if} \ h\equiv 1 \ (\text{mod}\ 3), \end{equation}
\begin{equation} \label{p2} \cfrac{h+1}{3}+\eta_p^{(k)},\quad k=0,1,2\quad \text{if} \ h\equiv 2\ (\text{mod} \ 3). \end{equation}
\begin{rem}
Using the expression of the SCP zeros provided by theorem \ref{root-scp}, we have a new explicit expression for cubic Gaussian periods involving $\cos$ and $\arctan$ functions.
\end{rem}
\begin{definition}
We call \emph{Lehmer--Ramanujan cubic polynomials} (LRCPs), a RCPs $\rho(h,s,x)$ such that $h \in \mathbb{Z}, 3\not | h$ and $\tau(h)$ is a prime number.
\end{definition}
Considering the relationships between zeros of Ramanujan and Shanks polynomials showed in the previous section, we can express LRCPs zeros in terms of cubic Gaussian periods.
\begin{theorem}
Let $\rho(h,s,x)$ be a LRCP and 
$$\theta_1(h,s,x)=\cfrac{x}{s}+\cfrac{h-1}{3},\quad \theta_2(h,s,x)=-\cfrac{x}{s}-\cfrac{h+1}{3},$$
$$ \gamma_1(h,s,x)=-\cfrac{1}{3}s(h-1-3x),\quad \gamma_2(h,s,x)=-\cfrac{1}{3}s(h+1+3x). $$

The zeros of $\rho(h,s,x)$ are

\begin{equation}\label{rootscrp1}
\gamma_1(h,s,\eta_p^{(k)}),\quad k=0,1,2 \quad \text{if} \quad h\equiv 1 \ (\text{mod} \ 3),\end{equation}
\begin{equation}\label{rootscrp2}
\gamma_2(h,s,\eta_p^{(k)}),\quad k=0,1,2 \quad \text{if} \quad h\equiv 2 \ (\text{mod}\ 3).\end{equation}
Therefore we have
\begin{equation}\label{idscrp1}
\rho(h,s,x)=s^3G_1(h,\theta_1(h,s,x))\quad\text{if} \quad h\equiv 1 \ (\text{mod}\ 3),\end{equation}
\begin{equation}\label{idscrp2}
\rho(h,s,x)=-s^3G_2(h,\theta_2(h,s,x))\quad\text{if} \quad h\equiv 2 \ (\text{mod}\ 3).\end{equation}
\end{theorem}
\begin{proof}
By corollary \ref{cor1}, the set $\zeta(h,s)$ corresponds to the set $-s\zeta(h,-1)$. Thus, by equations \eqref{p1} and \eqref{p2}, we obtain the equalities \eqref{rootscrp1} and \eqref{rootscrp2}. 
By comparison between the explicit expressions of $\rho(h,s,x)$, $G_1(h,x)$ and $G_2(h,x)$ we easily prove the identities \eqref{idscrp1} and \eqref{idscrp2}.
\end{proof}
Now, the Ramanujan identity \eqref{rama} provides us the following new relations involving the cubic Gaussian periods:
\small
\begin{equation} \label{gaussrama}
\left(\frac{h-1}{3}-\eta_p^{(0)}\right)^{\frac{1}{3}}+\left(\frac{h-1}{3}-\eta_p^{(1)}\right)^{\frac{1}{3}}+\left(\frac{h-1}{3}-\eta_p^{(2)}\right)^{\frac{1}{3}}=\left(6+h-3p^{\frac{1}{3}}\right)^{\frac{1}{3}}
\end{equation}
\normalsize
and
\small
\begin{equation} \label{gaussrama2}
\left(\frac{h+1}{3}+\eta_p^{(0)}\right)^{\frac{1}{3}}+\left(\frac{h+1}{3}+\eta_p^{(1)}\right)^{\frac{1}{3}}+\left(\frac{h+1}{3}+\eta_p^{(2)}\right)^{\frac{1}{3}}=\left(6+h-3p^{\frac{1}{3}}\right)^{\frac{1}{3}}
\end{equation}
\normalsize
when $h\equiv 1$ (mod $3$) and $h\equiv 2$ (mod $3$), respectively.
D. H. Lehmer and E. Lehmer \cite{project} studied differences of Gaussian periods, proving that their minimal polynomial is $x^3-px+p$. Let us consider
$$\delta_p^{(0)}=\eta_p^{(0)}-\eta_p^{(1)},\quad \delta_p^{(1)}=\eta_p^{(1)}-\eta_p^{(2)}, \quad \delta_p^{(2)}=\eta_p^{(2)}-\eta_p^{(0)}.$$
Using previous results, we provide a closed formula for these differences.
\begin{proposition}
If $p=\tau(h)$ is a prime number and $\theta_h=\cfrac{3\sqrt{3}}{3+2h}$, then
$$\delta_p^{(0)}=\pm2\sqrt{\frac{p}{3}}\cos \left(\frac{1}{3}\arctan\left( \frac{1}{\theta_h}\right) \right), \quad \delta_p^{(1)}=\pm2\sqrt{\frac{p}{3}}\cos \left(\frac{1}{3}\left(\arctan \left(\frac{1}{\theta_h}\right)+2\pi\right) \right),$$
$$\delta_p^{(2)}=\pm2\sqrt{\frac{p}{3}}\cos \left(\frac{1}{3}\left(\arctan\left( \frac{1}{\theta_h}\right)+4\pi\right) \right).$$
\end{proposition}
\begin{proof}
Thanks to the relation between Gaussian periods and Shanks polynomials, we know that $\delta_p^{(0)}=\alpha-\eta_{-1}(\alpha)$, where $\alpha$ is a zero of $\rho(h,-1,x)$. Moreover, by theorem \ref{root-scp} we have
$$\delta_p^{(0)}= \pm \frac{2}{3}\sqrt{p}\cos\left(\frac{1}{3}\arctan (\theta_h)\right)\mp \frac{2}{3}\sqrt{p}\cos\left(\frac{1}{3}\left(\arctan (\theta_h)+2\pi\right)\right)=$$
$$=\frac{\sqrt{p}}{3}\left(\pm2\cos\left(\frac{1}{3}\arctan (\theta_h)\right)\pm\cos\left(\frac{1}{3}\arctan (\theta_h)\right)\pm\sqrt{3}\sin\left(\frac{1}{3}\arctan( \theta_h)\right)\right)=$$
$$=\pm\sqrt{p}\cos\left(\frac{1}{3}\arctan (\theta_h)\right)\pm\sqrt{\frac{p}{3}}\sin\left(\frac{1}{3}\arctan (\theta_h)\right).$$
Recalling that, $\arctan x+ \arctan\frac{1}{x}=\frac{\pi}{2}$ if $x>0$ and $\arctan x+ \arctan\frac{1}{x}=-\frac{\pi}{2}$ if $x<0$, we have
\begin{eqnarray*}\delta_p^{(0)}&=&  \pm\cfrac{\sqrt{3}}{2}\sqrt{p}\cos \left(\frac{1}{3}\arctan\left( \frac{1}{\theta_h}\right)\right)\pm\cfrac{\sqrt{p}}{2}\sin \left(\frac{1}{3}\arctan\left( \frac{1}{\theta_h}\right)\right)\pm\\
                               &\pm&\cfrac{1}{2}\sqrt{\frac{p}{3}}\cos \left(\frac{1}{3}\arctan\left( \frac{1}{\theta_h}\right)\right)
\cfrac{\sqrt{3}}{2}\sqrt{\frac{p}{3}}\sin \left(\frac{1}{3}\arctan \left(\frac{1}{\theta_h}\right)\right)=\\
                               &=&\pm2\sqrt{\frac{p}{3}}\cos\left(\frac{1}{3}\arctan\left( \frac{1}{\theta_h}\right)\right).\end{eqnarray*}
Similarly, we obtain the thesis for $\delta_p^{(1)}$ and $\delta_p^{(2)}$.
\end{proof}

\section{On a Jefferey conjecture}
In this section we use our results on SCPs zeros to give a proof that the sequence $$\seqnum{A198636}=(3, 5, 13, 38, 117, 370, 1186,\ldots)$$ in OEIS \cite{oeis} is a linear recurrent sequence. We also prove a related conjecture due to L. E. Jefferey (Jan 21 2012).

Let $\alpha, \beta,$ and $\gamma$ be the zeros of the SCP $\rho(h,-1,x)$. We consider the integer sequence $(A(k,n))_{n=0}^{+\infty}$, where $k$ is a fixed integer and
\begin{equation} \label{seqa}  A(k,n)=\alpha^{kn}+\beta^{kn}+\gamma^{kn}.\end{equation}
Moreover, the companion matrix $M$ of the SCP is
$$M=\begin{pmatrix} 0 & 1 & 0 \cr 0 & 0 & 1 \cr 1 & 3+h & h  \end{pmatrix}.$$ Clearly, we have
$$A(1,n)=Tr(M^n),$$
and $(A(1,n))_{n=1}^{+\infty}$ is a linear recurrent sequence with characteristic polynomial $\rho(h,-1,x)$. Similarly, 
$$A(k,n)=Tr(M^{kn}),$$
and $(A(k,n))_{n=0}^{+\infty}$ is a linear recurrent sequence whose characteristic polynomial is $$x^3-A(k,1)x^2+B(k,1)x-1,$$ with zeros $\alpha^k, \beta^k$ and $\gamma^k .$
When $h=-1$, the SCP $\rho(-1,-1,x)$ has zeros
$$\bar\alpha=2\cos\left(\cfrac{2\pi}{7}\right),\quad \bar\beta=2\cos\left(\cfrac{4\pi}{7}\right),\quad \bar\gamma=2\cos\left(\cfrac{8\pi}{7}\right)$$ 
and the integer sequences arising from this polynomial are particularly interesting. These sequences have been deeply studied in some recent works of Witula and Slota \cite{Slota, quasifib, Witula2}.
\begin{lemma}\label{A(2,n)}
Let $(a_n)_{n=0}^{+\infty}$ be the linear recurrent sequence with characteristic polynomial $x^3-5x^2+6x-1$ and initial conditions $a_0=3$, $a_1=5$ and $a_2=13$,
then
$$a_n=2^{2n}\left( \cos^{2n}\left(\cfrac{2\pi}{7}\right) + \cos^{2n}\left(\cfrac{4\pi}{7}\right) + \cos^{2n}\left(\cfrac{8\pi}{7}\right) \right),$$
for $n=0,1,2,\ldots$.
\end{lemma}
\begin{proof}
The matrix
$$M^2=\begin{pmatrix} 0 & 0 & 1 \cr 1 & 2 & -1 \cr -1 & -1 & 3 \end{pmatrix}$$
has the same characteristic polynomial $x^3-5x^2+6x-1$ of the sequence $(A(2,n))_{n=1}^{+\infty}$. Moreover, evaluating the traces of the matrices $M^0$, $M^2$ and $M^4$, we observe that $A(2,0)=3$, $A(2,1)=5$ and $A(2,2)=13$, i.e., $a_n=A(2,n)$, for every $n=0,1,2,\ldots$. Thus, by equation \eqref{seqa}, we have the thesis.
\end{proof}
\begin{theorem}
The sequence  $(A_n)_{n=0}^{+\infty}=$\seqnum{A198636} in OEIS \cite{oeis} is a linear recurrent sequence corresponding to the sequence $(A(2,n))_{n=1}^{+\infty}$, in particular 
$$\begin{cases} A_0=3 \cr A_1=5 \cr A_2=13 \cr A_{n+3}=5A_{n+2}-6A_{n+1}+A_n.\end{cases}$$
Moreover
$$A_n=2^{2n}\left(\cos^{2n}\left(\frac{\pi}{7}\right)+\cos^{2n}\left(\frac{2\pi}{7}\right)+\cos^{2n}\left(\frac{3\pi}{7}\right)\right),$$
i.e., the Jefferey conjecture is true.
\end{theorem}
\begin{proof}
We recall that $A_n= \frac{w(6,2n)}{2}$ for all $n\geq0$, where $w(6,l)$ is the total number of closed walks of length $l$ on the graph $P_6$, the simple path with 6 points (vertices) and 5 lines (or edges). 
As explained in the comments related to the sequence \seqnum{A198632} in OEIS \cite{oeis}, if we consider the adjacency matrix $J_N$ of the graph $P_N$, we have $$w(N,l)=Tr(J_N^l)=\sum_{j=1}^{N}\lambda_{j}^{l},$$
where $\lambda_j$, $j=1,\ldots,N,$ are the eigenvalues of $J_N$.
Since the matrix $J_N$ is a Jacobi matrix, its characteristic polynomial is $S_N(x)=U_N(\frac{x}{2})$, where $U_N(x)$ is the second--kind Chebyshev polynomial of degree $N$  
(see, e.g., Mason and Handscomb \cite{Mason} for a survey on Chebyishev polynomials).
The zeros of $S_N(x)$ are $\lambda_j=2\cos\left(\frac{j\pi}{N+1}\right)$, $j=1,\ldots,N$ and in paricular we have
$$ w(6,2n)=\sum_{j=1}^{6}\left(2\cos\left(\frac{j\pi}{N+1}\right)\right)^{2n}=2^{(2n+1)}\left(\cos^{2n}\left(\frac{2\pi}{7}\right)+\cos^{2n}\left(\frac{4\pi}{7}\right)+\cos^{2n}\left(\frac{8\pi}{7}\right)\right).$$
Therefore, using the results of the previous lemma \ref{A(2,n)} , we clearly obtain  $$A_n=\frac{w(6,2n)}{2}=A(2,n) \quad n=0,1,\ldots $$ and the thesis immediately follows.
\end{proof}
\section{Acknowledgements}
The authors thank sincerely the referee for his appreciations, useful suggestions and comments about this paper.

2010 {\it Mathematics Subject Classification}:
Primary 11C08; Secondary 11B83.
 \emph{Keywords:} 
Ramanujan cubic polynomial.

(Concerned with sequences \seqnum{A005471}, \seqnum{A198632} and \seqnum{A198636}.)

\end{document}